\documentclass{amsart}
\usepackage{amsmath}
\usepackage{amsfonts}
\usepackage{latexsym}
\usepackage[all]{xypic}
\usepackage{lineno}
\newcommand{\cA}{\mathcal{A}}

\renewcommand{\cD}{\mathcal{D}}

\newcommand{\cF}{\mathcal{F}}
\newcommand{\cG}{\mathcal{G}}
\newcommand{\cK}{\mathcal{K}}
\renewcommand{\cL}{\mathcal{L}}
\newcommand{\cM}{\mathcal{M}}

\newcommand{\cO}{\mathcal{O}}
\newcommand{\cP}{\mathcal{P}}
\newcommand{\cS}{\mathcal{S}}
\newcommand{\cU}{\mathcal{U}}
\newcommand{\cV}{\mathcal{V}}
\newcommand{\bC}{\mathbb{C}}

\newcommand{\bZ}{\mathbb{Z}}

\newcommand{\Et}{E_\tau}

\DeclareMathOperator{\Supp}{Supp}
\DeclareMathOperator{\Sym}{Sym}
\DeclareMathOperator{\Fitt}{Fitt}

\DeclareMathOperator{\Pic}{Pic}

\DeclareMathOperator{\id}{id}

\DeclareMathOperator{\Hom}{Hom}
\DeclareMathOperator{\Ext}{Ext}

\DeclareMathOperator{\rank}{rank}
\DeclareMathOperator{\length}{length}

\DeclareMathOperator{\gr}{gr}

\newcommand{\xra}{\xrightarrow}

\newcommand{\ka}{{\mathcal A}}
\newcommand{\kc}{{\mathcal C}}
\newcommand{\kd}{{\mathcal D}}
\newcommand{\ko}{{\mathcal O}}
\newcommand{\ke}{{\mathcal E}}
\newcommand{\kf}{{\mathcal F}}

\newcommand{\kl}{{\mathcal L}}
\newcommand{\km}{{\mathcal M}}

\newcommand{\ku}{{\mathcal U}}
\newcommand{\kv}{{\mathcal V}}

\newcommand{\C}{{\mathbb C}}

\font\tencyr=wncyr10
\font\sevencyr=wncyr7
\font\fivecyr=wncyr5
\newfam\cyrfam
\def\cyr{\fam\cyrfam\tencyr}
\textfont\cyrfam=\tencyr
\scriptfont\cyrfam=\sevencyr
\scriptscriptfont\cyrfam=\fivecyr
\newtheorem{Th}{Theorem}
\newtheorem{Co}{Corollary}
\newtheorem{Pro}{Proposition}
\newtheorem{Def}{Definition}
\newtheorem{Lm}{Lemma}
\newtheorem{Rk}{Remark}
\numberwithin{equation}{section}
\NoCompileMatrices

\begin{document}
\title{Vector bundles on non-Kaehler elliptic principal bundles}
\author[V.Br\^inz\u anescu]{Vasile Br\^inz\u anescu}  
\address{"Simion Stoilow" Institute of Mathematics
of the Romanian Academy,
 P.O.Box 1-764, 014700 Bucharest, Romania}
\email{Vasile.Brinzanescu@imar.ro}
\thanks{ V.Br\^inz\u anescu was partially supported by CNCSIS contract 1189/2009-2011}
\author[A.D.Halanay]{ Andrei D.Halanay}
\address{ Faculty of Mathematics and Computer Science,\\
University of Bucharest
\newline Str. Academiei 14
\newline 010014 Bucharest,Romania}
\email{halanay@gta.math.unibuc.ro}
\author[G.Trautmann]{G\"unther Trautmann}
\address{ Universit\"at Kaiserslautern\\
  Fachbereich Mathematik\\
\newline  Erwin-Schr\"odinger-Stra{\ss}e
\newline D-67663 Kaiserslautern}
\email{trm@mathematik.uni-kl.de} 

\begin{abstract}

We study relatively semi-stable vector bundles and their moduli on 
non-K\"ahler principal elliptic bundles over compact complex manifolds of 
arbitrary dimension. The main technical
tools used are the twisted Fourier-Mukai transform and a spectral cover
construction. For the important example of such principal bundles, the numerical 
invariants of a 3-dimensional non-K\"ahler elliptic principal bundle over a 
primary Kodaira surface are computed.

\end{abstract}

\keywords{ non-K\"ahler principal elliptic bundles, Calabi-Yau type 
threefolds, holomorphic vector bundles, moduli spaces}
\subjclass{14J60, 32L05, 14D22, 14F05,32J17,32Q25}
\maketitle

%\title{Vector Bundles on non-K\"ahler elliptic principal bundles}
%\author{V.Br\^ \i nz\u nescu, A. Halanay, G. Trautmann}
%\tableofcontents
%\maketitle
\begin{section}{Introduction}
The study of vector bundles over elliptic fibrations has been a very active 
area of research in both
mathematics and physics over the past twenty years; in fact, there is by now 
a well understood theory
for projective elliptic fibrations (see for example \cite{Do}, \cite{DP}, 
\cite{F}, \cite{FMW}, \cite{BM},\cite{HM},
\cite{Br2}, etc.). However, not very much is known about the non-K\"ahler 
case; 
the study of rank 2 vector bundles on
non-K\"ahler elliptic surfaces is done in \cite{BrM1},\cite{BrM2}. 
In this article we study relatively semi-stable
vector bundles on non-K\"ahler principal elliptic bundles over complex
manifolds of arbitrary dimension with the  invariant $\delta\not = 0$. 
One of the motivations for
the study of vector bundles on non-K\"ahler elliptic n-folds comes from recent 
developments in superstring theory, where 
six-dimensional non-K\"ahler manifolds occur in the context of 
$\mathcal{N}=1$ supersymmetric heterotic and type II 
string compactifications with non-vanishing background $H-$ field; 
in particular most of the non-K\"ahler examples 
appearing in the physics literature so far are non-K\"ahler principal 
elliptic fibrations (see \cite{BBDG}, \cite{CCFLMZ}, \cite{GP}). 
There are also two classes of non-K\"ahler Calabi-Yau
type threefolds appearing in the mathematical and physical literature: 
one is due to M. Gross (privately communicated to us by A. C\u ald\u araru). 
Other examples appear in \cite{A1}, \cite{A2},\cite{Ku} and \cite{CDHPS}.
The main technical tools used are the twisted Fourier-Mukai transform, 
introduced by A.C\u ald\u araru (see \cite{Ca1}) and the spectral cover 
construction, see \cite{FMW}, \cite{Don}, \cite{BBRP}, \cite{HM}.

The paper is organized as follows.  In the second section we determine the 
structure of the relative Jacobian of a principal elliptic bundle as
a moduli space and find out that it is the product of the fiber with the 
basis. In the third and fourth sections, using the relative Jacobian, 
we adapt the construction of C\u ald\u araru, \cite{Ca2}, to our case, 
obtaining a twisted Fourier-Mukai
transform. Similar results were obtained in different settings by 
O.Ben-Bassat \cite{B-B} and I.Burban and B.Kreussler \cite{BK}.
In the fifths section using this transform and the associated 
spectral cover we prove that the moduli space of rank $n$, relatively 
semi-stable vector bundles is corepresented by the relative Douady space of 
length $n$ and relative dimension $0$ subspaces of the relative Jacobian,
see theorem 5.

After reviewing some background results on torus bundles from \cite{Ho} in the first appendix, 
in the second appendix we compute the numerical 
invariants (Hodge and Betti numbers) 
of a principal elliptic bundle over
a primary Kodaira surface and use them to distinguish the non-trivial elliptic 
bundles. These invariants are also of interest for
physicists working on heterotic string-theory models with non-K\"ahler
Calabi-Yau type threefolds as backgrounds.
\bigskip

\centerline{Acknowledgements}
\bigskip
Part of this paper was prepared during the stay at the Kaiserslautern 
Technical University of A.D. Halanay and
V. Br\^\i nz\u anescu with scholarships offered by the Alexander von 
Humboldt Foundation in the framework of the
Stability Pact. The first author expresses his gratidute to the 
Max-Planck-Institute f\"ur Mathematik in Bonn; part of this paper was 
prepared during his stay there.
\end{section}

\begin{section}{Line bundles on elliptic principal bundles}
In this section we shall be concerned with the study of the (coarse) 
moduli space of line bundles over a principal elliptic bundle $\pi:X\to S$
, where $S$ is a compact complex manifold, with fiber 
$E:=\Et:=\bC / \Lambda$ ($\Lambda=\bZ \oplus \tau \bZ$). 
Among the invariants of such bundles is the homomorphism
$\delta : H^1(E,\C)\to H^2(S,\C)$ which is the $d_2$-differential
$E^{0,1}_2\to E^{2,0}_2$ of the Leray spectral sequence of the sheaf
$\C_X$ with terms 
$E^{p,q}_2= H^p(S, R^q\pi_\ast\C_X)\cong H^p(S,\C)\otimes H^q(E,\C),$
see also  the first appendix for more invariants.

We make the assumption that $\delta \neq 0$. 
In particular, $X \to S$ does not have the topology of a product.
We should note here that if $S$ is K\"ahler, then X is non-K\"ahler if and 
only if $\delta \neq 0$, see \cite{Ho}.

We shall need in the sequel the following result of Deligne, \cite{De}, 
in the formulation of \cite[Prop.5.2]{Ho}.
\begin{Th}\label{D}
Let $X \to S$ be a principal elliptic bundle. Then the following 
statements are equivalent:
\begin{itemize}
\item[a)] The Leray spectral sequence for $\mathbb{C}_X$ degenerates 
at the $E_2-$level;
\item[b)] $\delta:H^1( E,\bC) \to H^2(X,\bC)$ is the zero map;
\item[c)] The restriction map $H^2(X,\bC) \to H^2( E,\bC)$ to a fibre takes a 
non-zero value in $H^{1,1}_{ E}$.
\end{itemize}
\end{Th}

In our case the preceding theorem has a very important consequence
\begin{Co}\label{C1}
Let $X \to S$ be a principal elliptic bundle with $S$ a compact 
complex manifold and $\delta \neq 0$. Then for any vector bundle $\cF$
 over $X$ and any $s \in S$ the bundle $\cF|_{X_s}$ has degree $0$.
\end{Co}

{\bf Proof.} Indeed let $r:H^2(X,\bC) \to H^2( E,\bC)$ be the restriction 
map. We have that $c_1(\cL|_{X_s})=r(c_1(\cL))=0$ by the theorem.
$\Box$
\vskip5mm

Let us recall now the definition of the Jacobian variety $J$ of a smooth curve 
$C$, see for instance \cite[IV.4]{Ha}.
Let $\Pic^0(C/T)$ for any analytic space $T$ denote the group 
$$\{\cL\in\Pic(T\times C)\; |\; \deg(\cL|_{\{t\}\times C})=0 
\text{ for any } t\in T\} / p^*\Pic(T),$$ 
where $p:T\times C \to T$ is the second projection.
The Jacobian variety of $C$ will be a variety $J$, 
together with an element $\cP \in \textrm{Pic}^0(C/J)$ such
that for any analytic space $T$ and any $\cM \in \textrm{Pic}^0(C/T)$ there 
is a unique morphism $f:T \to J$ such that 
$(f\times\id_C)^*\cP \sim \cM$ in 
$\textrm{Pic}^0(T\times C)$, i.e. $J$ represents the functor
$T\mapsto \Pic^0(C/T)$. 
It is well known that $J$ exists for any smooth curve $C$. When $C$ is an 
elliptic curve $E$ then $J= E^*$, the dual torus, and $\cP$ is called 
a Poincar\' e bundle. In this case $\cP$ is a line bundle over 
$ E^* \times  E$ such that 
$\cP|_{\{[\cL]\}\times E}\simeq\cL$, and $E^*\simeq\Pic^0( E).$  
We pass  now to the relative case for elliptic principal bundles.

\begin{Def}\label{reljac}
Let $X\xrightarrow{\pi}S$ be an elliptic principal bundle with typical 
fibre an elliptic curve $E_\tau$ and base $S$ a smooth manifold.  
Let $F:(An/S)^{op} \to (Sets)$ be the functor from the category 
of analytic spaces over $S$ to the category of sets, given, for any 
commutative diagram
\begin{equation} 
\xymatrix{X_T \ar[d]^\pi \ar[r] & X \ar[d]^\pi \\
T \ar[r] & S,}
\end{equation}
where $X_T:=X \times_S T$, by 
$$F(T):=\{\cL\; \textrm{invertible on}\; X_T\; |\; \textrm{deg}(\cL|_{X_{T,t}})
=0,\; \text{ for all}\quad t\in T\}/\sim, $$
where $\cL_1 \sim \cL_2$ if there is a line bundle $L$ on $T$ such 
that $\cL_1 \simeq \cL_2 \otimes \pi^*L$. 

A variety $J$ over $S$ will be called the relative Jacobian of $X$ if

(i) it corepresents the functor $F$, see \cite[Def. 2.2.1]{HL}, 
i.e. there is a natural transformation
$F\xra{\sigma}\Hom_S(-, J)$ and for any other variety $N/S$ with a natural 
transformation 
$F\xra{\sigma'}\Hom_S(-,N)$ there is a unique $S$-morphism $J\xra{\nu}N$
such that $\nu_*\circ\sigma=\sigma'$.

(ii) for any point $s\in S$ the map
$F(\{s\})\to \Hom_S(\{s\}, J)\simeq J_s$ is bijective. Then each fibre 
$J_s$ is the Jacobian of the fibre $X_s\simeq  E$.

\end{Def}

%By use of the relation $\sim$ we have also a convenient description of 
%$F(T)$ by taking isomorphism classes of line bundles, that is, 
%$F(T)=\{\cL \in \textrm{Pic}(X_T) \; | \deg(\cL|_{X_{T,t}})=0\;
%\text{ for all}\quad t\in T\}/\sim$ and $[\cL_1] \sim [\cL_2]$ if 
%$[\cL_1]=[\cL_2]\cdot \pi^*L$, with $L \in \Pic(T)$. 

If $X$ is projective, the existence of the relative Jacobian is well known,
because it can be identified with the coarse relative moduli space
of stable locally free sheaves of rank 1 and degree 0 on the fibres of $X$,
see \cite{HL}, \cite{Ca1}.
The relative Jacobian exists also in our non-K\"ahler case. It is just 
the product $S\times E^*$ and has the following special properties.
\vskip5mm

\begin{Th}\label{T3.2}

(i) The functor $F$ is corepresented by $J:=  S\times E^*$.

(ii) For any point $s\in S$ the map
$F(\{s\})\to \Hom_S(\{s\}, J)\simeq J_s\simeq  E^*$ is bijective.

(iii) The map $\sigma (T)$ is injective for any complex space $T$.

(iv) The functor $F$ is locally representable by $J=S\times E^*$,
i.e. if $U\subset S$ is a trivializing open subset, $\sigma(U)$
is bijective.
\end{Th}

For the proof we use the following Seesaw lemma and its Corollary.

\begin{Lm}\label{lemma4} Let $Y\xra{q} T$ be a principal elliptic bundle 
over a complex 
analytic space $T$ with fibre $ E$ and let $\km$ be an invertible sheaf on $Y$
such that $\km_t=\km|Y_t$ is trivial on any fibre of $q$. Then $L=q_*\km$
is locally free of rank 1 and $\km=q^*L$.
\end{Lm}

{\bf Proof.} Because the statements are local over $T$ and $Y$ is locally 
trivial, we may assume that $Y=T\times E$ and that $q$ is the first 
projection $p_1$.
We first note that the canonical homomorphism $p_1^\ast p_{1\ast}\cM\to\cM$ 
is an isomorphism because each $\km_t$ is trivial so that it is an 
isomorphism when restricted to a fibre of $p_1$.
In order to show that $p_{1\ast}\cM$ is invertible we consider the 
sheaf $\cM(a):=\cM\otimes p_2^\ast\cO_{ E}(a)$ for some $a\in E$.
Then $h^0(\cM_t(a))=1$ and $h^1(\cM_t(a))=0$, whereas $h^1(\cM_t)=1$.
Let 
$\phi^i(t): (R^i p_{1*} \cM(a))(t) \to H^i(X_t,\cM_t(a))$ denote the base 
change homomorphism. Because $\phi^1(t)=0$ it follows that 
$R^1 p_{1*}\cM(a)=0$, and by that that $L:= p_{1*}\cM(a)$ is invertible,
see e.g. \cite[Th.12.11]{Ha}.

Let $\bC(a)$ be the sky-scraper sheaf with stalk $\bC$ at $a\in E$.
Then $\cM\otimes p_2^*\bC(a)=\cM |_{T \times \{a\}}$ and we have the 
exact sequence
$$0 \to \cM \to \cM(a) \to \cM \otimes p_2^*\bC(a) \to 0,$$
giving rise to the long exact sequence
$$0\to p_{1*}\cM\to p_{1*}\cM(a)\xrightarrow{\alpha}
p_{1*}(\cM|_{S \times \{a\}})\to R^1 p_{1*}(\cM)\to 0.$$

Denoting by $\cA$ the image of $\alpha$ and pulling back to $y$, we obtain
the exact sequence 
$$p_1^*p_{1*}\cM\xrightarrow{\gamma}p_1^*p_{1*}\cM(a)\to p_1^*\cA\to 0.$$ 
The restriction of $\gamma$ to a fibre becomes the canonical map
$$H^0(X_t,\cM_t)\otimes\cO_{Y_t}\to H^0(Y_t,\cM_t(a))\otimes\cO_{Y_t},$$
which is an isomorphism. Hence $(p_1^*\cA)_t=0$ for any $t\in T$, and then
also $p_1^*\cA=0$ and finally $\cA=0.$
This proves that $p_{1*}\cM\cong p_{1*}\cM(a)=L$, and we have 
$\cM\cong p_1^*p_{1*}\cM\cong p_1^*L.$

One should note here that $p_1^\ast p_{1\ast}\cM(a)\to\cM(a)$ is not an 
isomorphism, having $\cM |_{T \times \{a\}}$ as its cokernel, because
$H^0(X_t,\cM_t(a)) \otimes \cO_{Y_t} \to \cM_t(a)$ is not an isomorphism.
$\Box$

\begin{Co}\label{C2}
Let $T$ be a complex space and 
$\xymatrix{T \times E_\tau \ar[r]^{\theta}_{\sim} & T \times E_\tau}$ 
an isomorphism of the form 
$\theta(t,\alpha)=(t,\alpha+ \lambda(t))$ with $\lambda:T \to  E$ 
a holomorphic map. 
Then for any invertible sheaf $\cL$ on $T \times E_\tau$ with $\deg(\kl_t)=0$
for any $t\in T$, there is an invertible sheaf $L$ on $T$ such that 
$\theta^* \cL \simeq \cL \otimes p_1^*L$. 
\end{Co}

{\bf Proof.} Because $\cL_s$ has degree $0$, 
$(\theta_s^*\cL_s) \otimes \cL_s^{-1}$ is trivial on any fiber 
(recall that a 
line bundle of degree $0$ is isomorphic with its pull-back via a translation, 
see \cite{BL}, \cite{M}). Now apply the Lemma to $\theta^*\kl\otimes\kl^{-1}$ 
$\Box$
\vskip5mm

{\bf Proof of Theorem \ref{T3.2}.}  Let $f:T \to S$ be a morphism of 
analytic spaces and let $\cL$ be a line bundle on $X_T$ . Let $\{S_i\}$ 
be an open 
cover of $S$ that gives a trivialization of the bundle $X$ 
(that is $X_i:=X|_{S_i} \simeq S_i \times E_\tau$). 
Taking the inverse image of this 
cover we get an open cover $\{T_i=f^{-1}S_i\}$ of $T$ with the same property. 
Let us denote $X_{T,i}:=X_T|_{T_i}$. There are trivializing maps
\begin{equation}
\xymatrix{X_{T,i}\ar[0,2]_{\sim}^{\theta_i} \ar[dr] & & T_i \times E_\tau 
\ar[dl] \\
& T_i\; . & }
\end{equation}
Let $\cL_i$ be the sheaf on $T_i\times E$ defined by  
$\theta_i^\ast\cL_i=\cL|_{X_{T,i}}$ such that  
$\cL_j \simeq (\theta_i \circ \theta_j^{-1})^*\cL_i$ over $T_{ij}$. 

Let now $\cP$ be a Poincare bundle on $E_\tau^* \times E_\tau$. Then for any 
$i$ we'll have a {\bf unique} morphism $\phi_i:T_i \to E^*_{
\tau}$ such that $\cL_i \sim (\phi_i \times \textrm{id})^*(\cP)$. 
Taking into account that $\theta_i\circ\theta_{j}^{-1} = id\times\theta_{ij}$
and  $\theta_{ij}$ acts by translations, the preceding corollary 
implies that $\cL_i \sim \cL_j$ on $T_{ij}.$ 
Therefore $\phi_i=\phi_j$ on $T_{ij}$. So we are given a global morphism 
$\phi :T \to E_\tau^*$.

Let now 
\begin{equation}
\xymatrix{T \ar[0,2]^-{\tilde{\phi}} \ar[dr] & & S \times E_\tau^* =J \ar[dl] \\& S &}
\end{equation}
be the corresponding map $\tilde{\phi}:=(f,\phi)$. This provides us 
with a map $F(T)$  $\xra{\sigma(T)} \textrm{Hom}_S(T,J)$. 
It is straightforward to check that $\sigma:F \to \textrm{Hom}_S(-,J)$ 
is a morphism of functors. The minimality of $J$ will be proved after the proof
of (iv).
\vskip3mm

ii) Property (ii) follows directly from the definition of the maps
$\sigma(\{s\})$.

iii) We show next that each map $\sigma(T):F(T) \to \textrm{Hom}_S(T,J)$ is  
injective.
For that let $\cL$ and $\cL'$ be two line bundles such that their 
respective maps $\phi$ and $\phi'$ are equal. 
We need to show that $\pi_*(\cL' \otimes \cL^{-1})=:L$ is locally-free. 

We have that $\phi_i=\phi_i'$ for any $i$. This implies that over 
$X_{T,i}$ we have 
\begin{equation}\label{e1}
\cL_i \sim \cL_i'.
\end{equation} 
Because of the relation (\ref{e1}) $\cL_t' \otimes \cL_t^{-1}$ is trivial 
for any $t \in T$. By the above corollary we obtain that 
$\pi_*(\cL' \otimes \cL^{-1})$ is locally free. 

iv) Let $U\subset S$ be trivializing for the bundle $X$ such that 
$X_U\simeq U\times E$. By the universal property of the dual torus
$ E^*$, the map $\sigma(U)$ as well as all the maps $\sigma(T)$
for $T\to U$ are bijective. 

Let now $F\xra{\sigma'} \Hom_S(-,N)$ be any natural transformation.
For any $s\in S$ there is the map 
$\nu_s:= \sigma'(\{s\})\circ\sigma(\{s\})^{-1}: J_s\to N_s$, thus defining
a map $\nu : J\to N$ over $S$. In order to show that $\nu$ is a morphism,
we just remark that $\nu_s$ is the restriction of the map 
$\sigma'(U\times  E^*)\circ\sigma(U\times  E^*)^{-1}(\id): J_U\to N_U$ 
for a trivializing open subset $U\subset S$ and any $s\in U$. 

Finally  $\nu_*\circ\sigma=\sigma'$ follows from (iv) of the theorem
and the fact that the functor $\Hom_S(-,N)$ is a sheaf, using a trivializing
covering for $X$ of $S$. This completes the proof of (i).
$\Box$
\vskip3mm

\begin{Rk} There is a very convenient description of $\sigma(S)$ as follows.
Let $\cL$ be an arbitrary line bundle over $X$. We know that $[\cL]\in F(S)$
by Corollary \ref{C1}. 
By the above proof $\phi:=\sigma(S)([\cL]):S \to S \times  E^*$
is given by $\phi(s)=(s,x)$ with $x=[\cL|_{X_s}]$.
\end{Rk}

It will follow from theorem \ref{T5.1} that the relative Jacobian $J=S\times E^*$
is only a coarse moduli space under our assumption on $X$.
However, by property (iv) of the theorem one can find a system
of local universal sheaves which will form a twisted sheaf in the next 
section as in \cite{Ca1}, chapter 4. 

\end{section}
\vskip1cm

\begin{section}{The twisted universal sheaf}

In the following we replace the relative Jacobian $J$ by $S\times E$ via the
canonical isomorphism between  $ E$ and $ E^*$. Then the local 
trivilalizations $X_i\xra{\theta_i}S_i\times E$ are at the same time 
isomorphisms between $X_i$ and $J_i:= S_i\times E$. The local universal 
sheaves $\ku_i$ on $X_{iJ}= J\times_S X_i=J_i\times_{S_i} X_i$ are then given
as pull backs of the universal sheaf 
$\ko_{E\times E}(\Delta)\otimes p_2^*\ko_E(-p_0)$ for the classical Jacobian
of the elliptic curve $E$, after fixing an origin $p_0\in E$ 
and where $\Delta$ is the diagonal.

 Denoting by $\rho_i$ the composition of maps
$$ X_{iJ}\xra{id\times\theta_i} J\times_S(S_i\times E)\simeq 
S_i\times E\times E\to E\times E,$$
and by $p_X$ the projection from $X_{iJ}$ to $X_i$, the local universal sheaf 
becomes
$$ \ku_i = \rho_i^*(\ko_{E\times E}(\Delta)\otimes p_2^*\ko_E(-p_0))
\simeq\ko_{X{iJ}}(\Gamma_i)\otimes p_X^*\ko_{X_i}(-s_i),$$
where $\Gamma_i$ is the inverse of the diagonal (or the graph of the map 
$\theta_i$) and $s_i$ is the section of $X_i$ corresponding to the reference 
point $p_0$ under the isomorphism $\theta_i$, see \cite{Ca1}, prop. 4.2.3.

To measure the failure of these bundles to glue to a 
global universal one let us consider the line bundles 
$\cM_{ij}:=\cU_i\otimes\cU_j^{-1}$ 
over $J\times_S X_{ij}$. Then the restriction of $\cM_{ij}$ to a fibre $X_s$ 
of the projection $J\times_s X_i\xra{q_i} J$ is trivial because both $\ku_j$
and $\ku_i$ restrict to isomorphic sheaves. By Lemma \ref{lemma4} there are
invertible sheaves $\kf_{ij}$ on $J_{ij}=S_{ij}\times E$ such that 
 $\cM_{ij}=q_i^*\kf_{ij}$. 

This collection of line bundles satisfies the following properties:
\begin{itemize}
\item[1.] $\cF_{ii}=\cO_{J_i}$;
\item[2.] $\cF_{ji}=\cF_{ij}^{-1}$;
\item[3.] $\cF_{ij}\otimes\cF_{jk}\otimes\cF_{ki}=:\cF_{ijk}$ is trivial, 
with trivialization induced by the canonical one of 
$\cM_{ij}\otimes\cM_{jk}\otimes\cM_{ki}$;
\item[4.] $\cF_{ijk}\otimes\cF_{jkl}^{-1}\otimes\cF_{kli}\otimes\cF_{lij}^{-1}$
 is canonically trivial. 
\end{itemize}
   
These conditions tell us that the collection $\{\cF_{ij}\}$ represents a gerbe 
(see \cite{DP})and gives rise to an element $\alpha\in H^2(J, \ko_J^*).$ 
More explicitly,
$\alpha$ is defined as follows. We may assume that the sheaves $\kf_{ij}$ are
already trivial with trivializations $a_{ij}:\ko_J\simeq\kf_{ij}$ over 
$J_{ij}$. 

If $c_{ijk}:\ko_J\simeq\kf_{ijk}$ is the isomorphism
which is induced by the canonical
trivialization of $\cM_{ij}\otimes\cM_{jk}\otimes\cM_{ki}$, then
\begin{equation}\label{cocycle}
a_{ij}\otimes a_{jk}\otimes a_{ki}= \alpha_{ijk}c_{ijk}
\end{equation}
with scalar functions $\alpha_{ijk}$ which then define a cocycle
for the sheaf $\ko_J^*$, thus defining the class 
$\alpha\in H^2(J,\ko_J^*)$, see see \cite{Ca1}, sect. 4.3. It is 
straightforward to prove:
\begin{Lm}\label{gluing} The sheaves $\ku_i$ can be glued to a 
global universal sheaf if and only if the class $\alpha=0$. 
\end{Lm}
The element $\alpha$ is related to the element $\xi\in H^1(S,\ko_S( E))$
which is defined by the cocycle of the elliptic bundle $X\to S$, using the
Ogg-Shafarevich group ${\cyr Sh}_S(J)$ of $J$, see \cite{Ca1}, section 4.4.
There is an exact sequence
$0\to\textrm{Br}(S)\to\textrm{Br}(J)\xra{\pi}{\cyr Sh}_S(J)\to 0,$
where $\textrm{Br}(S)\simeq H^2(J, \ko_J^*)$ is the analytic Brauer group of 
$S$ and ${\cyr Sh}_S(J)$ is isomorphic to $H^1(S,\ko_S( E))$ in our setting.
We have the 
 
\begin{Th}\label{T5.1}(\cite[Th 4.4.1]{Ca1}) $\xi = \pi(\alpha)$. \end{Th}

Because $\xi\not = 0$ in our case, $\alpha\not = 0$, and thus the local 
universal sheaves cannot be glued to a global universal sheaf by preserving
the bundle structure on the elliptic fibres.
\end{section}

\begin{section}{The twisted Fourier-Mukai transform}

The collection of local universal sheaves above can be considered 
as an $\alpha$-twisted sheaf with which one can define a Fourier-Mukai 
transform. Recall the definition of an $\alpha-$twisted sheaf on a complex
space or on an appropriate scheme $X$.
\begin{Def}
Let $\alpha \in C^2(\mathfrak{U},\cO_X^*)$ be a \v Cech 2-cocycle, given by an 
open cover $\mathfrak{U}=\{U_i\}_{i\in I}$ and
sections $\alpha_{ijk} \in \Gamma(U_i \cap U_j \cap U_k, \cO_X^*)$. 
An $\alpha$-twisted sheaf on $X$ will be a pair of families
$(\{\cF_i\}_{i\in I}, \{\varphi_{ij}\}_{i,j,\in I})$ with $\cF_i$ a sheaf of 
$\cO_X-$modules on $U_i$ and
$\varphi_{ij}:\cF_j|_{U_i \cap U_j} \to \cF_i |_{U_i \cap U_j}$  
isomorphisms such that
\begin{itemize}
\item $\varphi_{ii}$ is the identity for all $i \in I$.
\item $\varphi_{ij}=\varphi_{ji}^{-1}$ , for all $i,j \in I$.
\item $\varphi_{ij} \circ \varphi_{jk} \circ \varphi_{kl}$ is multiplication 
by $\alpha_{ijk}$ on $\cF_i |_{U_i \cap U_j \cap U_k}$ for all $i,j,k \in I$.
\end{itemize}
\end{Def}
It is easy to see that the coherent $\alpha$-twisted sheaves on X
make up an abelian category and thus give rise to a derived category
$\cD^\flat(X,\alpha).$  For further properties of $\alpha$-twisted 
sheaves, see \cite{Ca1}.

With the notation above, the family $(\ku_i)$ becomes a twisted sheaf $\ku$
w.r.t. the cocycle $p_J^*\alpha$ of the sheaf $\ko_{J\times_S X}^*$ as follows.
The trivializations $a_{ij}$ of the $\cF_{ij}$ induce isomorphisms 
$\phi_{ij}: \ku_j\simeq\ku_i$ which satisfy the definition of a twisted
sheaf because of identity \ref{cocycle}. We also need the dual $\kv$ of $\ku$
on $J\times_S X$ which locally over $S_i$ is given by 
$$ \cV_i = \rho_i^*(\ko_{E\times E}(-\Delta)\otimes p_2^*\ko_E(p_0))
\simeq\ko_{X{iJ}}(-\Gamma_i)\otimes p_X^*\ko_{X_i}(s_i).$$
It follows that $\kv_i$ is $\alpha^{-1}$-twisted. We let $\kv^0$ and $\ku^0$
denote the extensions of $\kv$ and $\ku$ to $J\times X$ by zero.

The following theorem supplies us with the main tool for the treatment
of the moduli spaces $M_X(n,0)$ of relatively semistable vector bundles on X
of rank $n$ and degree 0 on the fibres $X_s$ in section 5. 
It is an analog of theorem  \cite[Th.6.5.4]{Ca1}(and also \cite{Ca2}):
\begin{Th}
Let $X \xrightarrow{\pi}S$ be an elliptic principal fiber bundle. 
Let $\alpha \in \textrm{Br}(J)$ be the obstruction to 
the existence of the universal sheaf on $J\times_S X$ and let $\cU$ be the 
associated $p_J^*(\alpha)$-twisted universal sheaf on $J\times_S X$ 
with its dual $\kv$ as above.
 
\noindent Then the twisted Fourier-Mukai transform 
$\Psi: \cD^\flat(J,\alpha) \to \cD^\flat(X) $, given by 
$ \Psi(\cF):=Rp_{X*}(\cV^0 \otimes^L L{p_J}^*\cF)$ is an equivalence 
of categories, where $p_J$ and $p_X$ are the product projections
\begin{equation}
\xymatrix{J & J\times X\ar[l]_{p_J}\ar[r]^{p_X} & X}
\end{equation}
\end{Th}
Note here that $\cV^0 \otimes^L L{p_J}^*\cF)$ is an object
in the category of untwisted sheaves.

\begin{proof} The theorem follows from the C\u ald\u araru's version of the 
Bridgeland (Orlov, Mukai, etc.)-criterion 
(\cite{O},\cite{Br1}, \cite[Th.3.2.1]{Ca1}), applied to our case. Due to a private
communication this criterion works also in the case 
when $\alpha$ is not torsion\footnote{The authors are indebted to Andrei 
C\u ald\u araru for this information.}). It follows that
the functor $\Psi$ is fully faithful if and only if for each point 
$y \in J$ and its skyscraper sheaf $\bC(y)$, $\textrm{Hom}(\Psi (\bC(y))$,
 $\Psi (\bC(y)))=\bC$ and for any $y_1,y_2 \in J$ , 
$\textrm{Ext}^i(\Psi (\bC(y_1))$, $\Psi (\bC(y_2)))=0$, 
unless $y_1=y_2$ and $0 \leq i \leq \textrm{dim}(J)$. Moreover $\Psi$ is an 
equivalence of categories if and only if for any $y \in J$ we have 
$\Psi (\bC(y)) \otimes \omega_X \simeq \Psi (\bC(y))$. 
Note that in our case the canonical bundle is trivial so that the last 
condition is automatically satisfied.

\noindent In order to compute $\Psi(\bC(y)))$ for a point $y\in J$, let
$s\in S_i$ be the image in $S$ and consider $\kv^0\otimes p_J^*\bC(y)$. 
Its support is $(J\times_S X)\cap (\{y\}\times X)= \{y\}\times X_s$.
We may therefore identify $\kv$ with $\kv_i$ and obtain 
$\kv^0|\{y\}\times X\simeq\cO_{X_s}(-x+s_i(s))$, where $\theta_i(x)=y$ 
and $s_i$ denotes the 
local section of $X$ corresponding to $p_0$. 
Because $p_J^*\bC(y)= \cO_{y\times X_s},$ we obtain 
$\Psi(\bC(y)))\simeq p_{X*}\cO_{X_s}(-x+s_i(s))
\simeq i_{s*}\cO_{X_s}(-x+s_i(s))$
by the base change isomorphism for the inclusion $i_s$ which holds in this case
because $X_s$ is smooth as is the projection $p_X$, see \cite[Lemma 1.3]{BoO}
Using this, we conclude that 
$\Hom(\Psi (\bC(y)),\Psi (\bC(y)))=\bC$ and we proceed in the same 
way for $\Ext$.
\end{proof}

\begin{Rk}
        On can see this result in connection with section 6 of \cite{KO},
        since the element $\alpha\in \textrm{Br}(J)$ is not torsion.
\end{Rk}
In the sequel we shall work with the adjoint transform
$$\Phi(-)=Rp_{J*}(\cU^0 \otimes^L Lp_X^*(-))$$
 of $\Psi$, with kernel $\cU^0.$
It is the reverse equivalence, see \cite[8.4]{BM}, \cite{Hu}, \cite{BBH}
for the untwisted situation.

We need the following special cases of base change properties.
\begin{Pro}\label{bch}
For any $s\in S$ let $i_s:X_s \to X$ and $j_s:\{s\}\times E \to J$ be the 
natural inclusions. Then the canonical morphism of functors
\begin{equation}
\mathbf{L}j^*_s \circ \Phi \simeq \Phi_s \circ \mathbf{L}i_s^*,
\end{equation} is an isomorphism, 
where $\Phi_s$ is the classical Fourier-Mukai transform associated with the 
Poincar\`e bundle over $ E\times E$.
\end{Pro}
\begin{proof}Let $\tilde{j_s}$ denote the inclusion of $J_s\times X_s$ 
into $J\times_S X$ and let $p_{J_s}$ be the first projection of $J_s\times X_s$.
By \cite[Lemma 1.3]{BoO} $Lj_s^* p_{J*}= Rp_{J_s*}\tilde{j_s}^*$. Then 
$$Lj_s^*p_{J*}(\cV^0\otimes p_X^*\kf)\simeq 
Rp_{J_s*}\tilde{j_s}^*\cV^0\otimes p_X^*\kf\simeq 
(\cV|J_s\times X_s)\otimes p_{X_s}^*i_s^*\kf,$$
which implies the  formula. 
\end{proof}

The following definition is very usefull for dealing with the spectral
covers in the next section. 
\begin{Def}\label{defwit}(\cite{Muk})
We denote by $\Phi^i(\cF)$ the $i$-th term of the complex $\Phi(\cF)$. We say 
that the sheaf $\cF$ is $\Phi-$WIT$_i$ (the weak index theorem holds) if 
$\Phi^i(\cF) \neq 0$ and $\Phi^j(\cF) = 0$ for any $j \neq i$. Moreover if 
$\cF$ is WIT$_i$ and $\Phi^i(\cF)$ is locally 
free we say that $\cF$ is IT$_i$. 
\end{Def}

Consider now a rank $n$ vector-bundle $\cF$ over the principal elliptic bundle 
$X$ and denote its restriction to a fibre $X_s$ by $\cF_s.$ 
From Proposition \ref{bch} it follows that if 
$\cF_s$ is $\Phi_s$-WIT$_i$ for any $s$ then $\cF$ is $\Phi-$WIT$_i$.
\end{section}

\begin{section}{A spectral cover and vector bundles on $X$}

In this section we shall apply the twisted Fourier-Mukai transform 
to the moduli problem for rank-$n$ relatively semi-stable vector bundles on 
the principal elliptic bundle $X$. By Deligne's theorem (Theorem \ref{D}), 
the degree of the restriction $\cF_s$ of any vector 
bundle $\cF$ on $X$ is $0$ for any $s\in S$. Therefore we consider the set 
$MS_X(n,0)$ of rank-$n$ vector bundles on $X$ which are fibrewise semistable 
and of degree zero, together with its quotient 
$$M_X(n,0):= MS_X(n,0)/\sim $$
of equivalence classes, where two bundles are defined to be equivalent if
they are fibrewise S-equivalent.

Let us recall that a vector bundle $\ke$ on a smooth projective curve
is called semistable if for any proper subbundle $\ke'$, 
$$\deg(\ke')/\rank(\ke')\leq \deg(\ke)/\rank(\ke).$$
 For such bundles 
there is the standard notion of S-equivalence, see e.g. \cite{HL}.

It is well-known that the semistable vector bundles of degree zero on the 
elliptic curve $E$ are direct sums
\begin{equation}\label{atiyah}
\ke= \bigoplus_i\ka_{n_i}\otimes\ko_E(x_i-p_0),
\end{equation}
where the $\ka_n$ denote the indecomposable Atiyah bundles of degree zero
which are inductively defined by nontrivial extensions
$0\to\ko_E\to\ka_n\to\ka_{n-1}\to 0$ with $\ka_1=\ko_E$, see \cite{At}
\cite[Def. 1.12]{FMW} or \cite{T}. It follows that each such $\ke$ is 
S-equivalent to a direct sum $gr(\ke)= \oplus_j \ko_E(y_j-p_0)^{\oplus m_j}$
with pairwise distinct points $y_j$.

\begin{Pro}\label{wit1} (\cite{HM}, \cite{BBRP})
Let $\cF$ be a member of $MS_X(n,0).$ Then\\
(i) $\cF$ is $\Phi$-WIT$_1$.\\
(ii) for any $s\in S$ with $\cF_s$ as in \ref{atiyah}, the sheaf 
 $\Phi^1(\cF_s)$ is a skyscraper sheaf $\oplus_j\kc_j$ with
$\Supp(\kc_j)=\{-y_j\},$ (the point of the dual bundle 
$\ko_E(-y_j+p_0)\simeq\ko_E([-y_j]-p_0)$, $[-y_j]$ denoting the divisor of 
$-y_j\in E$) and $\length(\kc_j)=m_j$.
\end{Pro}

\begin{proof} The first part follows from
\cite{BBRP} Prop. 2.7 and Cor. 2.12. The second part follows by
direct computation and the base change property Proposition \ref{bch}.
\end{proof}

\begin{Rk}
        The proof shows that the sheaves $\Phi^1(\cF_s)$ and 
        $\Phi^1(\gr(\cF_s))$ are the same.
\end{Rk}

\begin{Rk}
The condition for $\cF_s$ to be $\Phi_s$-WIT$_1$ is even equivalent  
for $\cF_s$ to be of degree $0$ and semi-stable, see \cite{HM}.
\end{Rk}

For a $\Phi$-WIT$_1$-sheaf $\cF$ on $X$, we define the  
{\bf spectral cover} of $\cF$ as follows. 

\begin{Def}
Let $\cF$ be a WIT$_1$ sheaf on $X$. The spectral cover $C(\cF)$ of $\cF$ is 
the $0-$th Fitting subscheme of $J$ given by 
the Fitting ideal sheaf $Fitt_0(\Phi^1(\cF))$ of $\Phi^1(\cF)$.
\end{Def}
Because we work over a non-algebraic manifold and because the image of $\Phi$ 
is not in the derived category of 
coherent shaves, but in that of {\it twisted} sheaves, we need to prove that 
the Fitting scheme is well-defined in our case. But this follows from
the fact that the Fitting ideals are independent of the finite presentation
of the local sheaves $\kf_i$ of an $\alpha-$sheaf, see \cite[20.4]{E}.
Thus we have well-defined sheaves of ideals $\Fitt_l(\cF)$ 
given locally by the ideal sheaves $I_{p_i-l}(\kf_i)$ of minors of size
$p_i-l$ of the matrix $F_i$ of a local presentation
$\ko_J^{q_i}\xra{F_i}\ko_J^{p_i}\to\kf_i\to 0$ over the open set $U_i$.
This sheaf gives us an analytic subspace $V(\Fitt_l(\cF))$ called the $l$-th 
Fitting scheme by abuse of notation in the analytic category. 

By Proposition \ref{wit1} (ii), for a single fiber $X_s$, we are given a map 
$M_{X_s}(n,0)$ $\to\textrm{Sym}^n J_s$ 
from the moduli space of semistable vector bundles of rank $n$ and 
degree $0$ to the $n-$th symmetric power of the torus $J_s=\{s\}\times E$, 
defined by
$$\cF_s \mapsto \Sigma_j\; m_j(s,-y_j).$$

In this way we obtain a map from $M_X(n,0)$ to $S\times \Sym^nE$, where
$\Sym^nE:= E^n/{\mathfrak S}_n $ is the $n$-th symmetric power of $E$ as
the quotient of $E^n$ by the symmetric group ${\mathfrak S}_n$. Then
$S\times \Sym^nE$ is a complex manifold of dimension $n+2$ and can 
be thought of as the relative space of cycles of degree n in $E$.
We will show that this map is part of a transformation of functors with target 
$\Hom_S(-,S\times \Sym^nE)$ and that $S\times \Sym^nE$ 
corepresents the moduli functor $\km_X(n,0)$ for $M_X(n,0)$ defined
as follows.

For any complex space $T$ over $S$ let the set
$\cM_X(n,0)(T)$ be defined by
$$\cM_X(n,0)(T):=\cM\cS_X(n,0)(T)/\sim,$$
where $\cM\cS_X(n,0)(T)$ is the set of vector bundles on $X_T$ of rank $n$ 
and fibre degree 0, and where the equivalence relation $\cF\sim\cG$ is defined 
by $S-$equivalence
of the restricted sheaves $\cF_t$ and $\cG_t$ on the fibres $X_{Tt}$.
The functor property is then defined via pull backs.

We are going to describe the spectral cover as a functor below. For that let
$T\to S$ be a complex space over $S$ and let $\Phi_T$ be the Fourier-Mukai
transform for the product $J_T\times X_T$ with the pull back $\ku_T$ 
of $\ku$ as kernel. By \cite{BBRP}, Prop. 2.7 and Cor. 2.12, 
any bundle $\kf_T$ in $\cM\cS_X(n,0)(T)$ is also $\Phi_T-WIT_1$ and admits
a spectral cover $C(\kf_T)\subset T\times E$ defined by the Fitting ideal
$Fitt_0\Phi^1_T(\kf_T)$. 

\begin{Lm}\label{fl} If $T$ is reduced, then $C(\kf_T)$ is flat over $T$.
\end{Lm}

\begin{proof}  The fibres of $C(\kf_T)$ are finite of constant lenght $n$
as in the case of $S$ above. Because the projection to $T$ is surjective,
flatness follows from Douady's criterion in \cite{D}.
\end{proof}

\begin{Lm}\label{func} The spectral cover is compatible with base change:
For any morhism\\ $h:T^{'}\to T$ over $S$ and any bundle 
$\kf_T$ in $\cM\cS_X(n,0)(T)$, 
$$h^*C(\kf_T)\simeq C(h^*\kf_T)$$ 
\end{Lm}

\begin{proof}Because the fibres of the morphisms 
$p_J: J_T\times X_T\to J_T$ are 1-dimensional and the sheaves $\kf_T$
are locally free, base change holds for $R^1p_{J*}$, see \cite{BBRP} Prop.2.6.
When the induced map $J_{T^{'}}\to J_T$ is denoted by $h_J$, then
$$h_J^*\Phi^1_T(\kf_T)\simeq\Phi^1_{T^{'}}(h^*\kf_T)$$
for any  $\kf_T\in\cM\cS_X(n,0)(T)$. Since the Fitting ideals are also
compatible with base change, the claim follows.
\end{proof}

The spectral covers $C(\kf_T)$ lead us to consider the relative
Douady functors
$$\kd^n: (An/S)^{op}\to (Sets),$$
where $(An/S)$ denotes the category of complex analytic spaces over $S$
and where a set $\kd^n(T)$ for a morphism $T\to S$ is defined as the 
set of analytic subspaces $Z\subset T\times E$ which are flat over $T$
and have 0-dimensional fibres of constant lenght $n$. The Douady functor 
$\kd^n$ is represented by a complex space $D^n(S\times E/S)$ over $S$,
see \cite{P}. For a point $s\in S$, $\kd^n(\{s\})$ is the set of 
0-dimensional subspaces of length $n$ and can be identified with 
the symmetric product $\Sym^n(E)$ because it is well known that the
Hilbert-Chow morphism, in our case the Douady-Barlet morphism,
$\kd^n(\{s\})\to\{s\}\times \Sym^n(E)$ is an isomorphism for the smooth
curve $E$, see \cite{Ba} Ch.V. It is then easy to show that also the relative
Douady-Barlet morphism $D^n(S\times E/S)\to S\times\Sym^n(E)$ is an
isomorphism. This implies that for any complex space $T$ over $S$ there is
bijection  
\begin{equation}\label{biject}
\kd^n(T)\xra{\sim}\Hom_S(T,S\times\Sym^n(E)).
\end{equation}

One should note here that the behavior of families of cycles is more
difficult to describe than of those for the Douady space.

Let now $\kd^n_r$ resp.$\km_X(n,0)_r$ be the restriction of the functors 
$\kd^n$ and $\km_X(n,0)$ to the
category $(Anr/S)$ of reduced complex analytic spaces. By the Lemmas 
\ref{fl} and \ref{func} the spectral covers give rise to a transformation
of functors
\begin{equation}\label{transf}
\km_X(n,0)_r\xra{\gamma}\kd^n_r\simeq\Hom_S(-,S\times\Sym^n(E)),
\end{equation}
where for a reduced space $T$ over $S$ and for a class $[\kf_T]$ in 
$\km_X(n,0)(T)$ we have $\gamma(T)(\kf_T)=C(\kf_T).$ Note that by flatness 
$C(\cF_T)$ depends only on the equivalence class
of $\cF_T$. 
We are now able to present the main theorem which generalises theorem
\ref{T3.2}.
\begin{Th} Let $X\to S$ be an elliptic principal bundle over a compact
complex manifold $S$ of arbitrary dimension with invariant $\delta\not = 0.$
Then the spectral cover induces a transformation of functors\\ 
$\gamma : \km_X(n,0)_r\to \Hom_S(-,S\times\Sym^n(E))$ with the following
properties.

(i) The functor $\cM_X(n,0)_r$ is corepresented by $S\times\Sym^n(E)$ 
via the transformation $\gamma$, 

(ii) For any point $s\in S$ the induced map 
$M_{X_s}(n,0)\to\Sym^n(E)$ is bijective.

(iii) The map $\gamma(T)$ is injective for any reduced complex
space $T$ over $S$.

(iv) $\cM_X(n,0)_r$ is locally representable by $S\times\Sym^n(E),$
i.e. if $U\subset S$ is a trivializing open subset for $X$ and $T$ is a complex
space over $U$, then $\gamma(T)$ is bijective.
\end{Th}

\begin{proof} Property (ii) is clear by the construction of the functors.

\noindent We begin proving the injectivity in (iii). 
Let $[\cF_1],[\cF_2] \in \cM_X(n,0)(T)$ such that 
$\gamma(T)([\cF_1])=\gamma(T)[\cF_2])$. 
This implies that for every $t\in T$ the spectral covers of   
$(\cF_1)_t$ and $(\cF_2)_t$ are the same. If $(\cF_1)_t$ is S-equivalent
to $\oplus_j\ko_E(y_j-p_0)^{\oplus m_j}$, then 
$C((\cF_1)_t)=\Sigma_j m_j(-y_j)$ and visa versa by by Atiyah's classification.
Hence $(\cF_1)_t$ and $(\cF_2)_t$ are S-equivalent.
But this is precisely the equivalence relation for the classes 
$[\cF_1]$ and $[\cF_2]$.

\noindent To prove (iv), let $U\subset S$ be an open subset over which
$X$ is trivial and let $T\to U$ be a reduced complex space over $U$.
Then we can assume that $X_T = T\times E$. First we define a map 
$$\Hom_S(T, S\times E^n)\xra{b(T)}\cM_X(n,0)(T)$$ as follows.
Given a morphism $(p,f): T\to U\times E^n$ over $S$, let 
$f_\nu: T\to E$ be the $\nu$-th component of $f$. Let then  
$$\kl_\nu:= (f_\nu\times id)^*\ko_{E\times E}(-\Delta)\otimes p_2^*\ko_E(p_0)$$
on $T\times E$ be the pull back of the dual Poincare bundle.
Then the spectral cover of $\kl_{\nu, t}$ for any  point $t\in T$ 
consists of the point $f_\nu(t)\in E$. The map  $b(T)$ can now be defined
by $(p,f)\mapsto [\kl_1\oplus\cdots\oplus\kl_n]$. This map is obviously
$\mathfrak{S}_n$-equivariant and thus can be factorized through 
$\Hom_S(T, S\times\Sym^n(E))$, giving a map
$$\Hom_S(T, S\times\Sym^n(E))\xra{\beta(T)}\cM_X(n,0)(T).$$
By construction, $\beta(T)$ is an inverse of $\gamma(T).$

\noindent The proof of (i) is now analogous to that of (i) for Theorem 
\ref{T3.2}, using (iv).
\end{proof}
\end{section} 

\appendix\begin{section}{Invariants of torus bundles}
Let $M$ be an n-dimensional compact complex manifold, $T=V/ \Lambda$ 
an m-dimensional complex torus and $X\xrightarrow{\pi}M$ a principal bundle 
with fiber $T$. The theory of principal torus bundles is developed in great 
detail in \cite{Ho}; see also \cite{BU}. It is well known that 
such bundles are described by elements of $H^1(M,\cO_M(T))$, where 
$\cO_M(T)$ denotes the sheaf of local holomorphic maps from $M$ to $T$. 
Considering the exact sequence of groups
\begin{displaymath}
0 \to \Lambda \to V \to T \to 0
\end{displaymath}
and taking local sections we obtain the following exact sequence
\begin{displaymath}
0 \to \Lambda \to \cO_M \otimes V \to \cO_M(T) \to 0.
\end{displaymath}
Passing to the cohomology we have the long exact sequence
\begin{eqnarray*}
\cdots \to H^1(M,\Lambda) \to H^{0,1}_M \otimes V \to H^1(M,\cO_M(T)) 
\xrightarrow{c^\mathbb{Z}} \\
\xrightarrow{c^\mathbb{Z}} H^2(M,\Lambda)\to 
H^{0,2}_M \otimes V \to \cdots 
\end{eqnarray*}
By taking the image of the co-cycle defining the bundle via the map 
$c^\mathbb{Z}$ we obtain a characteristic class 
$c^\mathbb{Z}(X) \in H^2(M,\Lambda)=H^2(M,\mathbb{Z}) \otimes \Lambda$ 
and also a characteristic class $c(X) \in H^2(M,\bC) \otimes V$.

Concerning some important sheaves on $X$ we have (see \cite{Ho}): 
\begin{equation}\label{1}
\cK_X =\pi^*\cK_M, \quad R^i\pi_*\cO_X=\cO_M \otimes_\bC H^{0,i}(T) 
\end{equation}
and the exact sequence
\begin{equation}\label{2}
%0 \to \pi^*\Omega_M^1 \to \Omega^1_X \to \cO_X \otimes_\bC H^{1,0}(T) \to 0.
0\to\Omega_M^1\to\pi_\ast\Omega^1_X\to \cO_M \otimes_\bC H^{1,0}(T)
\to 0.
\end{equation}
All the informations concerning the topology of the bundle $X \to M$ are 
given by the following invariants
\begin{enumerate}
\item[a)] The exact sequence (\ref{2}) gives rise to an element 
$\gamma \in \textrm{Ext}^1(\cO_M \otimes H^{1,0}(T),
\Omega_M^1)=H^1(\Omega^1_M)\otimes H^{1,0}(T)^*$.
Thus $\gamma$ is a map $H^{1,0}(T)$ $ \to H^{1,1}(M)$.
\item[b)] The first non-trivial $d_2-$ differential in the Leray spectral 
sequence ($d_2:E_2^{0,1} \to E_2^{2,0}$) of the 
sheaf $\bC_X$. We obtain in this way a map $\delta:H^1(T,\bC) \to H^2(M,\bC)$. 
In the same way we may define the maps 
$\delta^\bZ:H^1(T,\bZ) \to H^2(M,\bZ).$\\
\item[c)] The first non-trivial $d_2-$differential in the Leray spectral 
sequence of $\cO_X,$ where 
$d_2:H^0(R^1\pi_*\cO_X) \to H^2(\pi_*\cO_X)$. Via
the identifications (\ref{1}) we get a map 
$\epsilon : H^{0,1}(T) \to H^{0,2}(M)$.
\item[d)] The characteristic classes $c^\bZ(X)$ and $c(X)$ defined above.
\end{enumerate}

These invariants are related by the following theorem of H\"ofer:
\begin{Th}
Let $X \xrightarrow{\pi}M$ be a holomorphic principal $T$-bundle. Then:
\begin{enumerate}
\item The Borel spectral sequence (\cite[Appendix Two by A.Borel]{Bo})\newline 
$^{p,q}E^{s,t}_2$ = $\sum H^{i,s-i}(M) \otimes H^{p-i,t-p+i}(T)$ degenerates 
on $E_3-$ level and the $d_2-$differential is given by $\epsilon$ and $\gamma$.
\item The Leray spectral sequence $E_2^{s,t}=H^s(M,\bC) \otimes H^t(T,\bC)$ 
degenerates on $E_3-$ level and the $d_2-$ 
differential is given by $\delta$.
\item Via the identification $H^1(T,\bZ) = \textrm{Hom}(\Lambda,\bZ)$ 
the characteristic class $c^\bZ$ and the map $\delta^\bZ$ coincide.
\item $\delta$ is determined by $\delta^\bZ$ via scalar extension.
\item If $H^2(M)$ has Hodge decomposition then $\delta$ determines 
$\epsilon$ and $\gamma$ and conversely.
\end{enumerate}
\end{Th}

In order to compute the Dolbeault cohomology of $X$ we need to use the Borel 
spectral sequence because the direct-image 
sheaves $R^j \pi_*\Omega_X^p$ are non-trivial for $p > 0$ 
and the Leray spectral sequence is more difficult to use.
\end{section}

\begin{section}{Invariants of elliptic principal bundles over surfaces}

In what follows we shall consider fiber bundles with basis $M$ a smooth 
complex surface  and fiber $T$ an elliptic curve. In this case 
a more detailed description is possible. Let $(1,\tau)$ be a basis of 
$\Lambda$, let $(dt,d\bar{t})$ be a basis of $H^1(T,\bC)$ 
given by the decomposition $H^1(T,\bC)=H^{1,0}(T) \oplus H^{0,1}(T)$ 
(another basis on $H^1(T,\bC)$ is given by the 
canonical coordinates $(dx_1,dx_2)$). Assume that 
$c^\bZ=a \otimes 1 + b \otimes \tau$, then $c=(a+b \cdot \tau)\otimes 1
=\eta \otimes 1$, with $\eta^{02}=0$. Then we have
\begin{eqnarray*}
\delta: dt \mapsto a+ \tau \cdot b = \eta \\
\quad d\bar{t} \mapsto a+\bar{\tau} \cdot b=\bar{\eta} \\
\epsilon : d\bar{t} \mapsto  (a+\bar{\tau} \cdot b)^{02}=\bar{\eta}^{02} \\
\gamma : dt \mapsto  (a+\tau \cdot b)^{11}=\eta^{11}.
\end{eqnarray*}
The only non-zero terms in the Borel spectral sequence are
\begin{equation}\label{3}
\xymatrix@C=0pt{ & H^{p-1,q-2} (M) \otimes H^{1,1}(T) \ar[dl]^\gamma 
\ar[dr]^{- \epsilon} & \\
 H^{p,q-1}(M)\otimes H^{0,1}(T) \ar[dr]^\epsilon & \oplus &  
 H^{p-1,q}(M)\otimes H^{1,0}(T) \ar[dl]^\gamma \\
& H^{p,q+1}(M) \otimes H^{0,0}(T). & } 
\end{equation}

From now on we shall be concerned with the case when $M$ has trivial 
canonical bundle. By (\ref{1}) this implies that $X$ 
also has trivial canonical bundle. The case when $M$ is K\"ahler was 
considered by H\"ofer. This leaves us with the case 
when $M$ is a primary Kodaira surface. We shall use the preceding diagram 
to compute the Hodge numbers for $X$ in this 
case. Recall that the Hodge diamond for a primary Kodaira together with 
Betti numbers is, see \cite[V.5.]{BPV}.
\begin{displaymath}
\begin{array}{ccccccc}
& &  1 &  & & & 1   \\
& 2 & & 1 &  & & 3  \\
1 & &  2 & & 1  & & 4 \\
& 1 & & 2 & & & 3\\
& & 1 & & & & 1
\end{array}
\end{displaymath}
Taking into account the Hodge numbers from above and the fact that all 
the Dolbeault groups of the elliptic curve that 
appear in (\ref{3}) are 1-dimensional we obtain the following Hodge diamond 
for $X$
\begin{equation}\label{4}
\begin{array}{ccccccccccccccccc}
& & & 1 & & & \\
& & 2-e & & 3-g & & \\
& 3-e &  & 6-g-h & & 2-g & \\
1 & & 5-g-h &  & 5-g-h & & 1 \\
& 2-g & & 6-g-h & & 3-e & &  \\
& & 2-g & & 1-e & & \\
& & & 1 & & & 
\end{array}
\end{equation}
where $e=\textrm{Rank}(\epsilon);g=\textrm{Rank}(\gamma)$ and $h$ is the 
rank of the map given by the multiplication with 
$\gamma(dt)-\epsilon(d\bar{t})$.

The first Betti number is given by 
$b_1(X) = b_1(M)+\textrm{dim Ker}(\delta)=3+2-d=5-d$, where 
$d=\textrm{Rank}(\delta)$.
To compute the Betti number $b_2(X)$ we shall use the Leray spectral 
sequence for the constant sheaf $\bC_X$. We have 
$E_2^{pq} = H^p(M,R^q\pi_*\bC_X)=H^P(M,\bC) \otimes H^q(T,\bC)$, 
the $d_2-$differential is determined by $\delta:E_2^{0,1}=
H^1(T,\bC) \to E_2^{20}=H^2(M,\bC)$ and the sequence degenerates at the 
$E_3$ level. In this case 
$E_\infty ^{02}=E_3^{02}=\textrm{Ker}(E_2^{02} \to E_2^{21})=
\textrm{Ker}(H^2(T,\bC)\to H^2(M,\bC)\otimes H^1(T,\bC))$, and so 
$E_3^{02}\simeq 0$ 
(we assumed that $\delta \neq 0$). Moreover, $E_3^{11}=\textrm{Ker}(E_2^{11} 
\to E_2^{30})=
\textrm{Ker}(H^1(M,\bC) \otimes H^1(T,\bC) \to H^3(M,\bC))$.
It follows that $\dim(E_3^{11})=6-d'$, where $d'$ is the rank of the map 
obtained by composing $\delta$ and the cup-product. 
Similarly, $\textrm{dim}(E_3^{20})=4-d$. We have the filtration 
$0 \subset F_2 \subset F_1 \subset F_0=H^2(X,\bC)$ 
associated with the spectral sequence (that is 
$F_2 \simeq E_\infty^{20}, F_1/F_2 \simeq E_\infty^{11}$ and $F_0/F_1 
\simeq E_\infty^{02}=0$). So we obtain an exact sequence 
$0 \to F_2 \to F_1 \to F_1/F_2 \to 0.$ From the above 
computations it follows that 
$b_2(X)=\textrm{dim}(E_\infty^{20})+\textrm{dim}(E_\infty^{11})=10-d-d'$.

For $b_3(X)$ we remark that in the Leray filtration 
$0 \subset F_3 \subset F_2 \subset F_1 \subset F_0=H^3(X,\bC)$ we have 
$F_1 = F_0$. This makes the things easier and by a 2-step computation 
we obtain that $b_3(X)=12-2d'$ so we can complete 
the table (\ref{4}) with
\begin{equation}\label{5}
\begin{array}{ccccccccccccccccc}
  1\\
  5-d\\
 10-d-d'\\
 12-2d'\\
  10-d-d'\\
 5-d\\
 1
\end{array}
\end{equation}
For comparison purpose we present also the results of H\"ofer 
(\cite[13.6,13.7]{Ho}). In the case $M$ is a torus we get 
 \begin{equation}
\begin{array}{ccccccccccccccccc}
& & & 1 & & & &  1\\
& & 5-f-g & & 3-e & & &  4\\
& 3-h &  & 8-f-g & & 3-e & &  8\\
1 & & 6-h &  & 6-h & & 1 &  10\\
& 3-e & & 8-f-g & & 3-h & &   8\\
& & 3-e & & 5-f-g & & & 4\\
& & & 1 & & & & 1,
\end{array}
\end{equation}
where $h:=\textrm{Rank}(H^{0,1}(M)\otimes H^{1,0}(T) 
\xrightarrow{\gamma(dt)\land \cdot} H^{1,2}(M)$ and
\newline
$f:= \textrm{Rank}(H^{1,0}(M)
\otimes H^{0,1}(T) \xrightarrow{\epsilon(d\bar{t})\land \cdot} H^{1,2}$. 
When $M$ is a $K3$ surface we have
\begin{equation}
\begin{array}{ccccccccccccccccc}
& & & 1 & & & &  1\\
& & 1-g & & 1-e & & &  0\\
& 1 &  & 20-g & & 1-e & &  20\\
1 & & 20 &  & 20 & & 1 &  42\\
& 1-e & & 20-g & & 1 & &   20\\
& & 1-e & & 1-g & & & 0\\
& & & 1 & & & & 1.
\end{array}
\end{equation}
\end{section}

\bibliographystyle{is-alpha}
\bibliography{Threef-biblio}

\end{document}